\documentclass{article}%
\usepackage{amsfonts}
\usepackage{amsmath}%
\usepackage{amssymb}%
\usepackage{graphicx}
\newtheorem{theorem}{Theorem}
\newtheorem{acknowledgement}[theorem]{Acknowledgement}

\newtheorem{corollary}[theorem]{Corollary}

\newtheorem{definition}[theorem]{Definition}

\newtheorem{proposition}[theorem]{Proposition}

\newenvironment{proof}[1][Proof]{\noindent\textbf{#1.} }{\ \rule{0.5em}{0.5em}}
\begin{document}

\title{\parbox{\linewidth}{\footnotesize\noindent } New properties of prolongations
of Linear connections on Weil bundles}
\author{B. V. NKOU\thanks{{\footnotesize vannborhen@yahoo.fr}} , B.G.R.
BOSSOTO\thanks{{\footnotesize bossotob@yahoo.fr}}, E.
OKASSA\thanks{{\footnotesize eugeneokassa@yahoo.fr}} ~}
\date{}
\maketitle

\begin{abstract}
Let $M$ be a paracompact smooth manifold, $A$ a Weil algebra and $M^{A}$ the
associated Weil bundle. If $\nabla$ is a linear connection on $M$, we give
equivalent definition and the properties of the prolongation $\nabla^{A}$ to
$M^{A}$ equivalent to the prolongation defined by Morimoto. When
$(M,\mathrm{g})$ is a pseudo-riemannian manifold, we show that the symmetric
tensor $\mathrm{g}^{A}$ of type $(0,2)$ defined by Okassa is nondegenerated.
At the end, we show that , if $\nabla$ is a Levi-Civita connection on
$(M,\mathrm{g})$, then $\nabla^{A}$ is torsion-free and $\mathrm{g}^{A}$ is
parallel with respect to $\nabla^{A}$.

\end{abstract}

\section{Introduction}

We recall that, in what follows we denote $A$, a local algebra (in the sense
of Andr\'{e} Weil) or simply Weil algebra, $M$ a smooth manifold, $C^{\infty
}(M)$ algebra of smooth functions on $M$ and $M^{A}$ the manifold of
infinitely near points of kind $A$ \cite{wei}. The triplet $(M^{A},\pi,M)$ is
a bundle called bundle of infinitely near points or simply Weil
bundle.\newline If $f:M\longrightarrow\mathbb{R}$ is a smooth function then
the application
\[
f^{A}:M^{A}\longrightarrow A,\xi\longmapsto\xi(f)
\]
is also a smooth function . The set, $C^{\infty}(M^{A},A)$ of smooth functions
on $M^{A}$ with values on $A,$ is a commutative algebra over $A$ with unit and
the application
\[
C^{\infty}(M)\longrightarrow C^{\infty}(M^{A},A),f\longmapsto f^{A}%
\]
is an injective homomorphism of algebras. Then, we have:
\begin{equation}
(f+g)^{A}=f^{A}+g^{A}\text{; }(\lambda\cdot f)^{A}=\lambda\cdot f^{A}%
\text{;}(f\cdot g)^{A}=f^{A}\cdot g^{A}.\nonumber
\end{equation}
The map
\[
C^{\infty}(M^{A})\times A\longrightarrow C^{\infty}(M^{A},A),(F,a)\longmapsto
F\cdot a:\xi\longmapsto F(\xi)\cdot a
\]
is bilinear and induces one and only one linear map
\[
\sigma:C^{\infty}(M^{A})\otimes A\longrightarrow C^{\infty}(M^{A},A).
\]
When $(a_{\alpha})_{\alpha=1,2,...,\dim A}$ is a basis of $A$ and when
$(a_{\alpha}^{\ast})_{\alpha=1,2,...,\dim A}$ is a dual basis of the basis
$(a_{\alpha})_{\alpha=1,2,...,\dim A}$, the application
\[
\sigma^{-1}:C^{\infty}(M^{A},A)\longrightarrow A\otimes C^{\infty}%
(M^{A}),\varphi\longmapsto\sum_{\alpha=1}^{\dim A}a_{\alpha}\otimes(a_{\alpha
}^{\ast}\circ\varphi)
\]
is an isomorphism of $A$-algebras. That isomorphism does not depend of a
choisen basis and the application
\[
\gamma:C^{\infty}(M)\longrightarrow A\otimes C^{\infty}(M^{A}),f\longmapsto
\sigma^{-1}(f^{A})\text{,}%
\]
is a homomorphism of algebras.\newline If $(U,\varphi)$ is a local chart of
$M$ with coordinate system $(x_{1},\,...,\,x_{n})$, the map
\[
\varphi^{A}:U^{A}\longrightarrow A^{n},\xi\longmapsto(\xi(x_{1}),\,...,\,\xi
(x_{n}))
\]
is a bijection from $U^{A}$ onto an open set of $A^{n}$. In addition, if
$(U_{i},\varphi_{i})_{i\in I}$ is an atlas of $M^{A}$, then $(U_{i}%
^{A},\varphi_{i}^{A})_{i\in I}$ is also an atlas of $M^{A}$ \cite{bo2}.

\subsection{Vector fields on $M^{A}$}

In \cite{nbo}, we gave another characterization of a vector field on $M^{A}$
through the above theorem and we also give a writing of a vector field on
$M^{A}$, in coordinate neighborhood system.

Thus,

\begin{theorem}
The following assertions are equivalent:

\begin{enumerate}
\item A vector field on $M^{A}$ is a differentiable section of the tangent
bundle $(TM^{A},\pi_{M^{A}},M^{A})$.

\item A vector field on $M^{A}$ is a derivation of $C^{\infty}(M^{A})$.

\item A vector field on $M^{A}$ is a derivation of $C^{\infty}(M^{A},A)$ which
is $A$-linear.

\item A vector field on $M^{A}$ is a linear map $X:C^{\infty}%
(M)\longrightarrow C^{\infty}(M^{A},A)$ such that%
\[
X(f\cdot g)=X(f)\cdot g^{A}+f^{A}\cdot X(g),\quad\text{for any}\,f,g\in
C^{\infty}(M)\text{.}%
\]

\end{enumerate}
\end{theorem}

We verify that the $C^{\infty}(M^{A},A)$-module $\mathfrak{X}(M^{A})$ of
vecvector field on $M^{A}$ is a Lie algebra over $A$.

\begin{theorem}
The map
\[
\mathfrak{X}(M^{A})\times\mathfrak{X}(M^{A})\longrightarrow\mathfrak{X}%
(M^{A}),(X,Y)\longmapsto\lbrack X,Y]=X\circ Y-Y\circ X
\]
is skew-symmetric $A$-bilinear and defines a structure of $A$-Lie algebra over
$\mathfrak{X}(M^{A})$.
\end{theorem}

In the following, we look at a vector field as a $A$-linear maps%
\[
X:C^{\infty}(M^{A},A)\longrightarrow C^{\infty}(M^{A},A)
\]
such that
\[
X(\varphi\cdot\psi)=X(\varphi)\cdot\psi+\varphi\cdot X(\psi),\quad\text{for
any}\,\varphi,\psi\in C^{\infty}(M^{A},A)
\]
that is to say
\[
\mathfrak{X}(M^{A})=Der_{A}[C^{\infty}(M^{A},A)].
\]

\subsection{Prolongations to $M^{A}$ of vector fields on $M$.}

\begin{proposition}
If $\ \theta:C^{\infty}(M)\longrightarrow C^{\infty}(M)$, is a vector field on
$M$, then there exists one and only one $A$-linear derivation%
\[
\theta^{A}:C^{\infty}(M^{A},A)\longrightarrow C^{\infty}(M^{A},A)\text{,}%
\]
such that $\ \theta^{A}(f^{A})=\left[  \theta(f)\right]  ^{A}$,for any $f\in
C^{\infty}(M)$. Thus, if $\theta,\theta_{1},\theta_{2}$ are vector fields on M
and if $f\in C^{\infty}(M)$, then we have:

\begin{enumerate}
\item
\[
\left(  \theta_{1}+\theta_{2}\right)  ^{A}=\theta_{1}^{A}+\theta_{2}%
^{A};\left(  f\cdot\theta\right)  ^{A}=f^{A}\cdot\theta^{A}\text{and }\left[
\theta_{1},\theta_{2}\right]  ^{A}=\left[  \theta_{1}^{A},\theta_{2}%
^{A}\right]  \text{.}%
\]

\end{enumerate}
\end{proposition}

\section{Prolongation of linear connections on Weil bundles}

In this section, if $\nabla$ \cite{hel} is a linear connection on $M$, we give
equivalent definition and the properties of the prolongation $\nabla^{A}$ to
$M^{A}$ equivalent to the prolongation $\overline{\nabla}$ defined by
Morimoto. When $(M,\mathrm{g})$ is a pseudo-riemannian manifold, we show that
the symmetric tensor $\mathrm{g}^{A}$ of type $(0,2)$ defined by Okassa is
nondegenerated. At the end, we show that , if $\nabla$ is a Levi-Civita
connection on $(M,\mathrm{g})$, then $\nabla^{A}$ is torsion-free and
$\mathrm{g}^{A}$ is parallel with respect to $\nabla^{A}$.

According \cite{nbo}, if $X:M^{A}\longrightarrow TM^{A}$ is a vector field on
$M^{A}$ and if $U$ is a coordinate neighborhood of $M$ with coordinate
neighborhood $(x_{1},...,x_{n})$, then there exists some functions $f_{i}\in
C^{\infty}(U^{A},A)$ for $i=1,...,n$ \ such that
\[
X_{|U^{A}}=\sum_{i=1}^{n}f_{i}\left(  \dfrac{\partial}{\partial x_{i}^{A}%
}\right)  ^{A}\text{.}%
\]

When $(U,\varphi)$ is local chart and $(x_{1},...,x_{n})$ his local coordinate
system. The map
\[
U^{A}\longrightarrow A^{n},\xi\longmapsto(\xi(x_{1}),...,\xi(x_{n})),
\]
is a diffeomorphism from $U^{A}$ onto an open set on $A^{n}$. As
\[
\left(  \dfrac{\partial}{\partial x_{i}}\right)  ^{A}:C^{\infty}%
(U^{A},A)\longrightarrow C^{\infty}(U^{A},A)
\]
is such that $\left(  \dfrac{\partial}{\partial x_{i}}\right)  ^{A}(x_{j}%
^{A})=\delta_{ij},$we can denote $\dfrac{\partial}{\partial x_{i}^{A}}=\left(
\dfrac{\partial}{\partial x_{i}}\right)  ^{A}$. If $v\in T_{\xi}M^{A}$, we can
write
\[
v=\sum_{i=1}^{n}v(x_{i}^{A})\dfrac{\partial}{\partial x_{i}^{A}}|_{\xi
}\ \text{.}%
\]
If $X\in\mathfrak{X}(M^{A})=Der_{A}[C^{\infty}(M^{A},A)]$, we have
\[
X_{|U^{A}}=\sum_{i=1}^{n}f_{i}\dfrac{\partial}{\partial x_{i}^{A}}.
\]
with $f_{i}\in C^{\infty}(U^{A},A)$ for $i=1,2,...,n.$

\subsection{Equivalent definitions of derivation laws in $\mathfrak{X}(M^{A})
$.}

In this subsection, we give the definitions of a derivation law in
$\mathfrak{X}(M^{A})=Der_{\mathbb{R}}[C^{\infty}(M^{A})]$ and of a derivation
law in $\mathfrak{X}(M^{A})=Der_{A}[C^{\infty}(M^{A},A)]$.\newline Let $R$ be
an algebra over a commutative field $\mathbb{K}$. We recall that, a derivation
law in a $R$-module $P$ is a map
\begin{align}
D:Der_{\mathbb{K}}(R)\longrightarrow End_{\mathbb{K}}(P),\nonumber
\end{align}
such that

\begin{enumerate}
\item $D$ is $R$-linear;

\item For any $d\in Der_{\mathbb{K}}(R)$, the $\mathbb{K}$-endomorphism
$D_{d}:P\longrightarrow P$ satisfies
\[
D_{d}(r\cdot p)=d(r)\cdot p+r\cdot D_{d}(p)
\]
for any $r\in R,\,\text{and any}\, p\in P$, see \cite{Kos}.
\end{enumerate}

We also recall that, a derivation law in the $C^{\infty}(M)$-module
$\mathfrak{X}(M)=Der_{\mathbb{R}}[C^{\infty}(M)$ module of vector fields on $M
$ is a map
\begin{align}
D: \mathfrak{X}(M)=Der_{\mathbb{R}}[C^{\infty}(M^{A})]\longrightarrow
End_{\mathbb{R}}[\mathfrak{X}(M)=Der_{\mathbb{R}}[C^{\infty}(M)]],\nonumber
\end{align}
such that

\begin{enumerate}
\item $D$ is $C^{\infty}(M)$-linear;

\item For any $\theta\in\mathfrak{X}(M)$, the $\mathbb{R}$-endomorphism
$D_{\theta}:\mathfrak{X}(M)\longrightarrow\mathfrak{X}(M)$ satisfies
\[
D_{\theta}(f\cdot\mu)=\theta(f)\cdot\mu+f\cdot D_{\theta}(\mu)
\]
for any $f\in C^{\infty}(M),\,\text{and any}\, \mu\in\mathfrak{X}(M^{A})$.
\end{enumerate}

That derivation law defines a linear connection on $M$, see \cite{pha}%
.\newline Now, we say:

\begin{definition}
A derivation law in $\mathfrak{X}(M^{A})=Der_{%
\mathbb{R}
}[C^{\infty}(M^{A})]$ is a map
\[
D:\mathfrak{X}(M^{A})=Der_{%
\mathbb{R}
}[C^{\infty}(M^{A})]\longrightarrow End_{%
\mathbb{R}
}\left[  \mathfrak{X}(M^{A})=Der_{%
\mathbb{R}
}[C^{\infty}(M^{A})]\right]  \text{,}%
\]
such that

\begin{enumerate}
\item $D$ is $C^{\infty}(M^{A})$-linear;

\item For any $X\in\mathfrak{X}(M^{A})$, the $%
\mathbb{R}
$-endomorphism $D_{X}:\mathfrak{X}(M^{A})\longrightarrow\mathfrak{X}(M^{A})$
satisfies
\[
D_{X}(F\cdot Y)=X(F)\cdot Y+F\cdot D_{X}(Y)
\]
for any $F\in C^{\infty}(M^{A})$, and any $Y\in\mathfrak{X}(M^{A})$.
\end{enumerate}
\end{definition}

\section*{Other definition.}

In what follows, we denote $\mathfrak{X}(M^{A})=Der_{A}[C^{\infty}(M^{A},A)]
$.\newline We denote $End_{A}[\mathfrak{X}(M^{A})]$ the set of $A$%
-endomorphisms of $\mathfrak{X}(M^{A})$ i.e the set of maps from
$\mathfrak{X}(M^{A})$ into $\mathfrak{X}(M^{A})$ which are linear over $A$.

\begin{proposition}
The set $End_{A}[\mathfrak{X}(M^{A})]$ is a $C^{\infty}(M^{A},A)$-module.
\end{proposition}

\begin{definition}
A derivation law in $\mathfrak{X}(M^{A})=$\ $Der_{A}[C^{\infty}(M^{A},A)]$. is
a map
\[
D:\mathfrak{X}(M^{A})\longrightarrow End_{%
\mathbb{R}
}\left[  \mathfrak{X}(M^{A})\right]  \text{,}%
\]
such that:

\begin{enumerate}
\item $D$ is $C^{\infty}(M^{A},A)$-linear;

\item For any $X\in\mathfrak{X}(M^{A})$, the $A$-endomorphism $D_{X}%
:\mathfrak{X}(M^{A})\longrightarrow\mathfrak{X}(M^{A})$ verifies
\[
D_{X}(\varphi\cdot Y)=X(\varphi)\cdot Y+\varphi\cdot D_{X}(Y)
\]
for any $\varphi\in C^{\infty}(M^{A})$, and any $Y\in\mathfrak{X}(M^{A})$.
\end{enumerate}
\end{definition}

\subsection{ The new statement of the Morimoto's prolongation of a linear
connection on $M$.}

\begin{theorem}
If $\nabla$ is a linear connection on $M$, then there exists one and only one
linear application%
\[
\nabla^{A}:\mathfrak{X}(M^{A})\longrightarrow End_{A}[\mathfrak{X}%
(M^{A})],X\longmapsto\nabla_{X}^{A}%
\]
such that
\[
\nabla_{\theta^{A}}^{A}\eta^{A}=\left(  \nabla_{\theta}\eta\right)
^{A}\text{,}%
\]
for any $\theta,\eta\in\mathfrak{X}(M)$.
\end{theorem}

\begin{proof}
If $X\in\mathfrak{X}(M^{A})=Der_{A}[C^{\infty}(M^{A},A)$, then
\[
X(f^{A})=\sum_{\alpha=1}^{\dim A}X^{^{\prime}}(a_{\alpha}^{\ast}\circ
f^{A})\cdot a_{\alpha}=\sum_{\alpha=1}^{\dim A}X(a_{\alpha}^{\ast}\circ
f^{A})\cdot a_{\alpha}%
\]
with $X^{^{\prime}}\in\mathfrak{X}(M^{A})=Der_{A}[C^{\infty}(M^{A})]$.\newline
Let
\[
\overline{\nabla}:\mathfrak{X}(M^{A})=Der_{\mathbb{R}}[C^{\infty}%
(M^{A})]\longrightarrow End_{\mathbb{R}}\left[  \mathfrak{X}(M^{A}%
)=Der_{\mathbb{R}}[C^{\infty}(M^{A})]\right]
\]
be the Morimoto's prolongation to $M^{A}$ of the linear connection $\nabla$ on
$M$. We denote
\[
\nabla^{A}:\mathfrak{X}(M^{A})=Der_{A}[C^{\infty}(M^{A},A)]\longrightarrow
End_{A}\left[  \mathfrak{X}(M^{A})=Der_{A}[C^{\infty}(M^{A},A)]\right]
\]
the same derivation law in $\mathfrak{X}(M^{A})=Der_{A}[C^{\infty}(M^{A},A)]
$. Thus for any $\theta,\,\eta\in\mathfrak{X}(M)$, we have:
\begin{align*}
\left[  \nabla_{\theta^{A}}^{A}\eta^{A}\right]  (f^{A})=\sum_{\alpha=1}^{\dim
A}\left[  \nabla_{\theta^{A}}^{A}\eta^{A}\right]  ^{^{\prime}}(a_{\alpha
}^{\ast}\circ f^{A})\cdot a_{\alpha}=  & \sum_{\alpha=1}^{\dim A}\left[
\nabla_{(\theta^{A})^{^{\prime}}}^{A}(\eta^{A})^{^{\prime}}\right]
(a_{\alpha}^{\ast}\circ f^{A})\cdot a_{\alpha}\\
& =\sum_{\alpha=1}^{\dim A}\left[  (\nabla_{\theta}\eta)^{A}\right]
^{^{\prime}}(a_{\alpha}^{\ast}\circ f^{A})\cdot a_{\alpha}\\
& =\sum_{\alpha=1}^{\dim A}\left[  (\nabla_{\theta}\eta)^{A}\right]
(a_{\alpha}^{\ast}\circ f^{A})\cdot a_{\alpha}\\
& =\left[  (\nabla_{\theta}\eta)^{A}\right]  (f^{A})\text{,}%
\end{align*}
for any $f\in C^{\infty}(M)$, hence
\[
\nabla_{\theta^{A}}^{A}\eta^{A}=(\nabla_{\theta}\eta)^{A}.
\]

\end{proof}

\subsubsection{Torsion of $\nabla^{A}$.}

When $\nabla$ is a linear connection on $M$, we denote $T_{\nabla}$ the
torsion of $\nabla$. \ 

\begin{proposition}
If $\nabla$ is a linear connection on $M$, then the torsion of $\nabla^{A}$%
\[
T_{\nabla^{A}}:\mathfrak{X}(M^{A})\times\mathfrak{X}(M^{A})\longrightarrow
\mathfrak{X}(M^{A}),(X,Y)\longmapsto=\nabla_{X}^{A}Y-\nabla_{Y}^{A}%
X-[X,Y]\text{,}%
\]
is a skew-symmetric $C^{\infty}(M^{A},A)$-bilinear.
\end{proposition}

\begin{proof}
\begin{enumerate}
\item For all vector fields $X,Y,Z\in\mathfrak{X}(M^{A})$, we have:
\begin{align*}
T_{\nabla^{A}}(X+Y,Z)  & =\nabla_{(X+Y)}^{A}Z-\nabla_{Z}^{A}(X+Y)-[X+Y,Z]\\
& =\nabla_{X}^{A}Z+\nabla_{Y}^{A}Z-\nabla_{Z}^{A}(X)-\nabla_{Z}^{A}%
(Y)-[X,Z]-[Y,Z]\\
& =\nabla_{X}^{A}Z-\nabla_{Z}^{A}(X)-[X,Z]+\nabla_{Y}^{A}Z-\nabla_{Z}%
^{A}(Y)-[Y,Z]\\
& =T_{\nabla^{A}}(X,Z)+T_{\nabla^{A}}(Y,Z).
\end{align*}

\item For any vector field $X\in\mathfrak{X}(M^{A})$, we have:
\begin{align*}
T_{\nabla^{A}}(X,X)  & =\nabla_{X}^{A}X-\nabla_{X}^{A}X-[X,X]\\
& =0.
\end{align*}

\item For any vector fields $X\in\mathfrak{X}(M^{A})$ and for any $\varphi\in
C^{\infty}(M^{A},A)$, we have%
\begin{align*}
T_{\nabla^{A}}(X,\varphi\cdot Y)  & =\nabla_{X}^{A}\varphi\cdot Y-\nabla
_{\varphi\cdot Y}^{A}(X)-[X,\varphi\cdot Y]\\
& =X(\varphi)\cdot Y+\varphi\cdot\nabla_{X}^{A}Y-\varphi\cdot\nabla_{Y}%
^{A}X-X(\varphi)\cdot Y-\varphi\cdot\lbrack Y,X]\\
& =\varphi\cdot\nabla_{X}^{A}Y-\varphi\cdot\nabla_{Y}^{A}X-\varphi\cdot\lbrack
Y,X]\\
& =\varphi\cdot\left(  \nabla_{X}^{A}Y-\nabla_{Y}^{A}X-[Y,X]\right) \\
& =\varphi\cdot T_{\nabla^{A}}(X,Y)\text{.}%
\end{align*}

\end{enumerate}

Therefore the torsion $T_{\nabla^{A}}$ is skew-symmetric $C^{\infty}(M^{A},A)$-bilinear.
\end{proof}

\begin{proposition}
For any $X,Y\in\mathfrak{X}(M^{A})$, and if $U$ is coordonate neighborhood of
$M$, then
\[
T_{\nabla_{|U^{A}}^{A}}(X_{|U^{A}},Y_{|U^{A}})=\left[  T_{\nabla^{A}%
}(X,Y)\right]  _{|U^{A}}\text{.}%
\]

\end{proposition}

\begin{proposition}
If $\nabla$ is a linear connection on $M$, then%
\[
T_{\nabla^{A}}(\theta^{A},\eta^{A})=\left[  T_{\nabla}(\theta,\eta)\right]
^{A}%
\]
for any $\theta,\eta\in\mathfrak{X}(M)$.
\end{proposition}

\begin{proof}
For any $\theta,\eta\in\mathfrak{X}(M)$, we have:
\begin{align*}
T_{\nabla^{A}}(\theta^{A},\eta^{A})  & =\nabla_{\theta^{A}}^{A}\eta^{A}%
-\nabla_{\eta^{A}}^{A}\theta^{A}-[\theta^{A},\eta^{A}]\\
& =[\nabla_{\theta}\eta]^{A}-[\nabla_{\eta}\theta]^{A}-[\theta,\eta]^{A}\\
& =\left(  \nabla_{\theta}\eta-\nabla_{\eta}\theta-[\theta,\eta]^{A}\right) \\
& =[T_{\nabla}(\theta,\eta)]^{A}\text{.}%
\end{align*}

\end{proof}

\begin{corollary}
If the linear connection $\nabla$ is torsion-free, then $\nabla^{A}$ is also torsion-free
\end{corollary}

\begin{proof}
Let $X,Y$ be two vector fields $M^{A}$ and $U$ a coordinate neighborhood of
$M^{A}$. Then
\[
X_{|U^{A}}=\sum_{i=1}^{n}f_{i}\dfrac{\partial}{\partial x_{i}^{A}};Y_{|U^{A}%
}=\sum_{j=1}^{n}g_{j}\dfrac{\partial}{\partial x_{j}^{A}}%
\]
and, we have:
\begin{align*}
\left[  T_{\nabla^{A}}(X,Y)\right]  _{|U^{A}}  & =T_{\nabla_{|U^{A}}^{A}%
}(X_{|U^{A}},Y_{|U^{A}})\\
& =T_{\nabla_{|U^{A}}^{A}}\left(  \sum_{i=1}^{n}f_{i}\dfrac{\partial}{\partial
x_{i}^{A}},\sum_{j=1}^{n}g_{j}\dfrac{\partial}{\partial x_{j}^{A}}\right) \\
& =\sum_{ij=1}^{n}f_{i}g_{j}T_{\nabla_{|U^{A}}^{A}}\left(  \dfrac{\partial
}{\partial x_{i}^{A}},\dfrac{\partial}{\partial x_{j}^{A}}\right) \\
& =\sum_{ij=1}^{n}f_{i}g_{j}T_{\nabla_{|U^{A}}^{A}}\left(  \left(
\dfrac{\partial}{\partial x_{i}}\right)  ^{A},\left(  \dfrac{\partial
}{\partial x_{j}}\right)  ^{A}\right) \\
& =\sum_{ij=1}^{n}f_{i}g_{j}\left[  T_{\nabla_{|U}}\left(  \dfrac{\partial
}{\partial x_{i}},\dfrac{\partial}{\partial x_{j}}\right)  \right]  ^{A},
\end{align*}
as $\nabla$ is torsion-free that is to say $T_{\nabla}=0$, hence $\left[
T_{\nabla^{A}}(X,Y)\right]  _{|U^{A}}=0$. Consequently
\[
T_{\nabla^{A}}=0.
\]

\end{proof}

\subsection{Prolongation of the Levi-Civita connection.}

In this subsection we consider $(M,\mathrm{g})$ a pseudo-riemannian manifold,
in what follows we study the prolongation of connections to $M^{A}$ deduce
from the Levi-Civita connection on $M$.

\begin{proposition}
\cite{oka} Let $\mathrm{g}:\mathfrak{X}(M)\times\mathfrak{X}(M)\longrightarrow
C^{\infty}(M)$ be a symmetric tensor of type $(0,2)$ on $M$. There exists one
and only one symmetric tensor $\mathrm{g}^{A}$ of type $(0,2)$ on $M^{A}$ with
value in $A$ such that $\mathrm{g}^{A}\left(  a\cdot\eta^{A},b\cdot\theta
^{A}\right)  =ab\cdot\left[  \mathrm{g}\left(  \eta,\theta\right)  \right]
^{A}$ for any $a,b\in A$ and $\eta,\theta\in\mathfrak{X}(M)$.
\end{proposition}

Following \cite{bo1}, we state:

\begin{proposition}
When $(M,\mathrm{g})$ a pseudo-riemannian manifold, then there exists one and
only one $C^{\infty}(M^{A},A)$-nondegenerated symmetric bilineat form
\[
\mathrm{g}^{A}:\mathfrak{X}(M^{A})\times\mathfrak{X}(M^{A})\longrightarrow
C^{\infty}(M^{A},A)
\]
such that for any vector fields $\eta$ and $\theta$ on $M$,
\[
\mathrm{g}^{A}\left(  \eta^{A},\theta^{A}\right)  =\left[  \mathrm{g}\left(
\eta,\theta\right)  \right]  ^{A}%
\]
where $\eta^{A}$ and $\theta^{A}$mean prolongations to $M^{A}$ of vector
fields $\eta$ and $\theta$.
\end{proposition}

\begin{proof}
It is a matter here to show only the nondegeneracy of $\mathrm{g}^{A}$, the
proof is in the same way as in \cite{bo1}.
\end{proof}

Therefore $g^{A}$ is a pseudo-riemannian manifold on $M^{A}$ and confers to
$M^{A}$ the structure of pseudo-riemannian manifold.

\begin{proposition}
For any $X\in\mathfrak{X}(M^{A})$, the map
\[
\nabla_{X}^{A}\mathrm{g}^{A}:\mathfrak{X}(M^{A})\times\mathfrak{X}%
(M^{A})\longrightarrow C^{\infty}(M^{A},A)
\]
such that
\[
\nabla_{X}^{A}\mathrm{g}^{A}(Y,Z)=X\left[  \mathrm{g}^{A}(Y,Z)\right]
-\mathrm{g}^{A}\left(  \nabla_{X}^{A}(Y),Z\right)  -\mathrm{g}^{A}\left(
Y,\nabla_{X}^{A}Z\right)
\]
for any $Y,Z\in\mathfrak{X}(M^{A})$ is a symmetric $C^{\infty}(M^{A}%
,A)$-bilinear form.
\end{proposition}

\begin{proof}
\begin{enumerate}
\item For any $X,Y\in\mathfrak{X}(M^{A})$, we have:
\begin{align*}
\nabla_{X}^{A}\mathrm{g}^{A}(Y,Z)  & =X\left[  \mathrm{g}^{A}(Y,Z)\right]
-\mathrm{g}^{A}\left(  \nabla_{X}^{A}(Y),Z\right)  -\mathrm{g}^{A}\left(
Y,\nabla_{X}^{A}Z\right) \\
& =X\left[  \mathrm{g}^{A}(Z,Y)\right]  -\mathrm{g}^{A}\left(  Z,\nabla
_{X}^{A}(\varphi\cdot Y)\right)  -\mathrm{g}^{A}\left(  \nabla_{X}%
^{A}Z,\varphi\cdot Y\right) \\
& =\nabla_{X}^{A}\mathrm{g}^{A}(Z,Y),
\end{align*}
hence $\nabla_{X}^{A}\mathrm{g}^{A}$ is symmetric.

\item Let $Y_{1},Y_{2}$ and $Z$ be the vector fields in $\mathfrak{X}(M^{A})$,
we have:
\begin{align*}
\nabla_{X}^{A}\mathrm{g}^{A}\left(  Y_{1}+Y_{2},Z\right)   & =X\left[
g^{A}(Y_{1}+Y_{2},Z)\right]  -\mathrm{g}^{A}\left(  \nabla_{X}^{A}(Y_{1}%
+Y_{2}),Z\right)  -\mathrm{g}^{A}\left(  Y_{1}+Y_{2},\nabla_{X}^{A}Z\right) \\
& =X\left[  \mathrm{g}^{A}(Y_{1},Z)+g^{A}(Y_{2},Z)\right]  -\mathrm{g}%
^{A}\left(  \nabla_{X}^{A}Y_{1}+\nabla_{X}^{A}Y_{2},Z\right)  -\mathrm{g}%
^{A}\left(  Y_{1},\nabla_{X}^{A}Z\right) \\
& -\mathrm{g}^{A}\left(  Y_{2},\nabla_{X}^{A}Z\right) \\
& =X\left[  g^{A}(Y_{1},Z)\right]  +X\left[  g^{A}(Y_{2},Z)\right]
-g^{A}\left(  \nabla_{X}^{A}Y_{1},Z\right)  -\mathrm{g}^{A}\left(  \nabla
_{X}^{A}Y_{2},Z\right) \\
& -\mathrm{g}^{A}\left(  Y_{1},\nabla_{X}^{A}Z\right)  -\mathrm{g}^{A}\left(
Y_{2},\nabla_{X}^{A}Z\right) \\
& =X\left[  \mathrm{g}^{A}(Y_{1},Z)\right]  -\mathrm{g}^{A}\left(  \nabla
_{X}^{A}Y_{1},Z\right)  -\mathrm{g}^{A}\left(  Y_{1},\nabla_{X}^{A}Z\right)
+X\left[  \mathrm{g}^{A}(Y_{2},Z)\right] \\
& -\mathrm{g}^{A}\left(  \nabla_{X}^{A}Y_{2},Z\right)  -\mathrm{g}^{A}\left(
Y_{2},\nabla_{X}^{A}Z\right) \\
& =\nabla_{X}^{A}\mathrm{g}^{A}\left(  Y_{1},Z)+\nabla_{X}^{A}\mathrm{g}%
^{A}(Y_{2},Z\right)  \text{.}%
\end{align*}

\item Let $Y$ and $Z$ the vector fields in $\mathfrak{X}(M^{A})$ and
$\varphi\in C^{\infty}(M^{A},A),$ we have:%
\begin{align*}
\nabla_{X}^{A}\mathrm{g}^{A}\left(  \varphi\cdot Y,Z\right)   & =X\left[
\mathrm{g}^{A}\left(  \varphi\cdot Y,Z\right)  \right]  -\mathrm{g}^{A}\left(
\nabla_{X}^{A}(\varphi\cdot Y),Z\right)  -\mathrm{g}^{A}\left(  \varphi\cdot
Y,\nabla_{X}^{A}Z\right) \\
& =X(\varphi)\cdot\mathrm{g}^{A}(Y,Z)+\varphi\cdot X\left[  \mathrm{g}%
^{A}(Y,Z)\right]  -\mathrm{g}^{A}\left(  X(\varphi)\cdot Y,Z\right)
+\varphi\cdot\mathrm{g}^{A}\left(  \nabla_{X}^{A}Y,Z\right) \\
& -\varphi\cdot\mathrm{g}^{A}\left(  Y,\nabla_{X}^{A}Z\right) \\
& =X(\varphi)\cdot\mathrm{g}^{A}(Y,Z)+\varphi\cdot X\left[  \mathrm{g}%
^{A}(Y,Z)\right]  -X(\varphi)\cdot\mathrm{g}^{A}(Y,Z)-\varphi\cdot
\mathrm{g}^{A}\left(  \nabla_{X}^{A}Y,Z\right) \\
& -\varphi\cdot\mathrm{g}^{A}\left(  Y,\nabla_{X}^{A}Z\right) \\
& =\varphi\cdot X\left[  \mathrm{g}^{A}(Y,Z)\right]  -\varphi\cdot
\mathrm{g}^{A}\left(  \nabla_{X}^{A}Y,Z\right)  -\varphi\cdot\mathrm{g}%
^{A}\left(  Y,\nabla_{X}^{A}Z\right) \\
& =\varphi\cdot\nabla_{X}^{A}\mathrm{g}^{A}(Y,Z).
\end{align*}

\end{enumerate}

Therefore, the map $\nabla_{X}^{A}\mathrm{g}^{A}$ is a symmetric $C^{\infty
}(M^{A},A)$-bilinear form.
\end{proof}

\begin{proposition}
If $\nabla$ is a linear connection on the pseudo-riemannian manifold
$(M,\mathrm{g})$, then we have:%
\[
\nabla_{\theta^{A}}^{A}\mathrm{g}^{A}\left(  \mu_{1}^{A},\mu_{2}^{A}\right)
=\left[  \nabla_{\theta}\mathrm{g}\left(  \mu_{1},\mu_{2}\right)  \right]
^{A}%
\]
for any $\theta,\mu_{1},\mu_{2}\in\mathfrak{X}(M^{A})$.
\end{proposition}

\begin{proof}
for any $\theta,\mu_{1},\mu_{2}\in\mathfrak{X}(M^{A})$, we have:
\begin{align*}
\nabla_{\theta^{A}}^{A}\mathrm{g}^{A}(\mu_{1}^{A},\mu_{2}^{A})  & =\theta
^{A}\left[  \mathrm{g}^{A}(\mu_{1}^{A},\mu_{2}^{A})\right]  -\mathrm{g}%
^{A}\left(  \nabla_{\theta^{A}}^{A}\mu_{1}^{A},\mu_{2}^{A}\right)
-\mathrm{g}^{A}\left(  \mu_{1}^{A},\nabla_{\theta^{A}}^{A}\mu_{2}^{A}\right)
\\
& =\theta^{A}\left[  (\mathrm{g}(\mu_{1},\mu_{2}))^{A}\right]  -[\mathrm{g}%
\left(  \nabla_{\theta}\mu_{1},\mu_{2}\right)  ]^{A}-[\mathrm{g}\left(
\mu_{1},\nabla_{\theta}\mu_{2}\right)  ]^{A}\\
=  & [\theta(\mathrm{g}(\mu_{1},\mu_{2}))]^{A}-[\mathrm{g}\left(
\nabla_{\theta}\mu_{1},\mu_{2}\right)  ]^{A}-[\mathrm{g}\left(  \mu_{1}%
,\nabla_{\theta}\mu_{2}\right)  ]^{A}\\
=  & [\theta\lbrack\mathrm{g}(\mu_{1},\mu_{2})]-\mathrm{g}(\nabla_{\theta}%
\mu_{1},\mu_{2})-\mathrm{g}(\mu_{1},\nabla_{\theta}\mu_{2})]^{A}\\
=  & [\nabla_{\theta}\mathrm{g}(\mu_{1},\mu_{2})]^{A}.
\end{align*}

\end{proof}

\begin{proposition}
For any $X,Y,Z\in\mathfrak{X}(M^{A})$, and if $U$ is coordinate neighborhood
of $M$, then
\[
\left[  \left(  \nabla_{|U^{A}}^{A}\right)  _{|U^{A}}\mathrm{g}_{|U^{A}}%
^{A})\right]  \left(  X_{|U^{A}},\,Y_{|U^{A}}\right)  =\left[  \nabla_{X}%
^{A}\mathrm{g}^{A}(Y,Z)\right]  _{|U^{A}}\text{.}%
\]

\end{proposition}

\begin{corollary}
If $\nabla$ is the Levi-Civita connection on the pseudo-riemannian manifold
$(M,g)$, then we have:
\[
\nabla_{X}^{A}\mathrm{g}^{A}=0
\]
for any $X\in\mathfrak{X}(M^{A})$.
\end{corollary}

\begin{proof}
Let $X,Y,Z$ be vector fields $M^{A}$ and $U$ a coordinate neighborhood of
$M^{A}$. Then
\[
X_{|U^{A}}=\sum_{i=1}^{n}f_{i}\dfrac{\partial}{\partial x_{i}^{A}}%
;\,Y_{|U^{A}}=\sum_{j=1}^{n}g_{j}\dfrac{\partial}{\partial x_{j}^{A}%
};\,Z_{|U^{A}}=\sum_{k=1}^{n}h_{k}\dfrac{\partial}{\partial x_{k}^{A}}.
\]
Thus, we have:
\begin{align*}
\left[  \nabla_{X}^{A}\mathrm{g}^{A}(Y,Z)\right]  _{|U^{A}}  & =\left[
(\nabla_{|U^{A}}^{A})_{X_{|U^{A}}})\mathrm{g}_{{|U^{A}}}^{A}\left(  Y_{|U^{A}%
},Z_{|U^{A}}\right)  \right]  \\
& =\left(  (\nabla_{|U^{A}}^{A})_{\left(  \sum_{i=1}^{n}f_{i}\dfrac{\partial
}{\partial x_{i}^{A}}\right)  }\mathrm{g}_{{|U^{A}}}^{A}\right)  \left(
\sum_{j=1}^{n}g_{j}\dfrac{\partial}{\partial x_{j}^{A}},\sum_{k=1}^{n}%
h_{k}\dfrac{\partial}{\partial x_{k}^{A}}\right)  \\
& =\sum_{ijk=1}^{n}f_{i}g_{j}h_{k}\left(  (\nabla_{|U^{A}}^{A})_{\left(
\dfrac{\partial}{\partial x_{i}^{A}}\right)  }\mathrm{g}_{{|U^{A}}}%
^{A}\right)  \left(  \dfrac{\partial}{\partial x_{j}^{A}},\dfrac{\partial
}{\partial x_{k}^{A}}\right)  \\
& =\sum_{ijk=1}^{n}f_{i}g_{j}h_{k}\left(  (\nabla_{|U^{A}}^{A})_{\left(
\dfrac{\partial}{\partial x_{i}}\right)  ^{A}}\mathrm{g}_{|U^{A}}^{A}\right)
\left(  \left(  \dfrac{\partial}{\partial x_{j}}\right)  ^{A},\left(
\dfrac{\partial}{\partial x_{k}}\right)  ^{A}\right)  \\
& =\sum_{ijk=1}^{n}f_{i}g_{j}h_{k}\left(  \left(  (\nabla_{|U})_{\left(
\dfrac{\partial}{\partial x_{i}}\right)  }\mathrm{g}_{|U}\right)  \left(
\dfrac{\partial}{\partial x_{j}},\dfrac{\partial}{\partial x_{k}}\right)
\right)  ^{A}\text{.}%
\end{align*}

As $\nabla$ is the Levi-Civita connection, then $\nabla_{\theta}\mathrm{g}=0$,
hence $\left[  \nabla_{X}^{A}\mathrm{g}^{A}(Y,Z)\right]  _{|U^{A}}=0$. It
follows that,
\[
\nabla_{X}^{A}\mathrm{g}^{A}=0.
\]

\end{proof}

\begin{theorem}
If $\nabla$ is a Levi-Civita connection on a pseudo-riemannian manifold
$(M,\mathrm{g})$, then $\nabla^{A}$ verifies the following properties:

\begin{enumerate}
\item $T_{\nabla^{A}}=0$;

\item $\nabla_{X}^{A}\mathrm{g}^{A}=0$ for any $X\in\mathfrak{X}(M^{A})$.
\end{enumerate}
\end{theorem}

\begin{proof}
The proof is deduced from the corollary \ref{pp1} and corollary \ref{pp5}.
\end{proof}

Thus $\nabla^{A}$ is a Levi-Civita connection on the pseudo-riemannian
manifold $(M^{A},\mathrm{g}^{A})$.\newline

\begin{acknowledgement}
: The first author thanks Deutscher Akademischer Austauschdientst (DAAD) for
their financial support.
\end{acknowledgement}


\begin{thebibliography}{99}                                                                                               %
\bibitem {bos1}Bossoto, B.G.R., \textit{Structures de Jacobi sur une
vari\'{e}t\'{e} des points proches}, Math. Vesnik. 62, 2 (2010), 155-167.

\bibitem {bo2}Bossoto, B.G.R., Okassa, E., \textit{Champs de vecteurs et
formes diff\'{e}rentielles sur une vari\'{e}t\'{e} des points proches},
Archivum mathematicum (BRNO), Tomus44 (2008),159-171.

\bibitem {hel}Helgason, S., \textit{Differential Geometry and symmetric
spaces,} New York; Academic Press, 1962.

\bibitem {Kos}Koszul, J.L., Ramanan, S., \textit{Lectures On Fibre Bundles and
Differential Geometry}, Tata Institute of Fundamental Research, Bombay 1960.

\bibitem {mor}Morimoto, A., \textit{Prolongation of connections to bundles of
infinitely near points}, J. Diff. Geom, 11 (1976), 479-498.

\bibitem {nbo}bb

\bibitem {oka}Okassa, E., \textit{Rel\`{e}vements de structures symplectiques
et pseudo-riemanniennes \`{a} des vari\'{e}t\'{e}s de points proches,} Nagoya
Math. J.115 (1989), 63-71.

\bibitem {Oka1}Okassa, E., \textit{Prolongement des champs de vecteurs \`{a}
des vari\'{e}t\'{e}s des points prohes,} Ann. Fac. Sci. Toulouse Math. VIII
(3) (1986-1987), 346-366.

\bibitem {pha}Pham-Ham-Mau-Quan, F., \textit{Introduction \`{a} la
g\'{e}om\'{e}trie des vari\'{e}t\'{e}s diff\'{e}rentiables,} Dunod Paris, 1969.

\bibitem {wei}Weil, A., \textit{Th\'{e}orie des points proches sur les
vari\'{e}t\'{e}s diff\'{e}rentiables}, Colloq. G\'{e}om. Diff. Strasbourg
(1953), 111-117
\end{thebibliography}
\end{document}